\documentclass[a4paper,12pt]{amsart}
\usepackage{amssymb}
\usepackage{cite}
\usepackage{ifthen}
\usepackage[dvips]{graphicx}
\nonstopmode \numberwithin{equation}{section}
\usepackage{amssymb}
\usepackage{ifthen}
\usepackage{graphicx}
\usepackage{amsmath}
\usepackage[T1]{fontenc} 
\usepackage[utf8]{inputenc}
\usepackage[usenames,dvipsnames]{color}
\usepackage{color}
\usepackage[english]{babel}
\usepackage{fancyhdr}
\usepackage{fancybox}
\usepackage{tikz}

\theoremstyle{plain}
\newtheorem{prop}{Proposition}

\newtheorem{conj}{Conjecture}

\theoremstyle{definition}
\newtheorem{defi}{Definition}[section]

\newtheorem{cor}{Corollary}[section]
\newtheorem{thm}{Theorem}[section]

\newtheorem{lem}{Lemma}[section]
\newtheorem{prob}{Problem}
\newtheorem{rem}{Remark}[section]

\newcounter{minutes}
\divide\time by 60
\newcounter{hours}
\multiply\time by 60
\addtocounter{minutes}{-\time}

\newcounter {own}
\def\theown {\thesection       .\arabic{own}}

\newenvironment{pf}[1][]{%
	\vskip 3mm
	\noindent
	\ifthenelse{\equal{#1}{}}%
	{{\slshape Proof. }}%
	{{\slshape #1.} }%
}%
{\qed\bigskip}

\newcounter{alphabet}

\newcommand{\real}{{\operatorname{Re}\,}}

\def\be{\begin{equation}}
	\def\ee{\end{equation}}

\newcommand{\bee}{\begin{enumerate}}
	\newcommand{\eee}{\end{enumerate}}

\newcommand{\blem}{\begin{lem}}
	\newcommand{\elem}{\end{lem}}
\newcommand{\bthm}{\begin{thm}}
	\newcommand{\ethm}{\end{thm}}
\newcommand{\bcor}{\begin{cor}}
	\newcommand{\ecor}{\end{cor}}
\newcommand{\beg}{\begin{examp}}
	\newcommand{\eeg}{\end{examp}}
\newcommand{\begs}{\begin{examples}}
	\newcommand{\eegs}{\end{examples}}

\newcommand{\bdefn}{\begin{defn}}
	\newcommand{\edefn}{\end{defn}}

\newcommand{\bprob}{\begin{prob}}
	\newcommand{\eprob}{\end{prob}}
\newcommand{\bei}{\begin{itemize}}
	\newcommand{\eei}{\end{itemize}}

\newcommand{\bcon}{\begin{conj}}
	\newcommand{\econ}{\end{conj}}
\newcommand{\bcons}{\begin{conjs}}
	\newcommand{\econs}{\end{conjs}}
\newcommand{\bprop}{\begin{prop}}
	\newcommand{\eprop}{\end{prop}}
\newcommand{\br}{\begin{rem}}
	\newcommand{\er}{\end{rem}}
\newcommand{\brs}{\begin{rems}}
	\newcommand{\ers}{\end{rems}}
\newcommand{\bo}{\begin{obser}}
	\newcommand{\eo}{\end{obser}}
\newcommand{\bos}{\begin{obsers}}
	\newcommand{\eos}{\end{obsers}}
\newcommand{\bpf}{\begin{pf}}
	\newcommand{\epf}{\end{pf}}
\newcommand{\ba}{\begin{array}}
	\newcommand{\ea}{\end{array}}
\newcommand{\beq}{\begin{eqnarray}}
	\newcommand{\beqq}{\begin{eqnarray*}}
		\newcommand{\eeq}{\end{eqnarray}}
	\newcommand{\eeqq}{\end{eqnarray*}}

\newcommand{\ra}{\rightarrow}

\begin{document}

\title{The sharp refined Bohr-Rogosinski inequalities for certain classes of harmonic mappings}

\author{Molla Basir Ahamed}
\address{Molla Basir Ahamed, 
	Department of Mathematics, 
	Jadavpur University, 
	Kolkata-700032, 
	West Bengal,India.}
\email{mbahamed.math@jadavpuruniversity.in}


\subjclass[{AMS} Subject Classification:]{Primary 30C45, 30C50, 30C80}

\def\thefootnote{}
\footnotetext{ {\tiny File:~\jobname.tex,
printed: \number\year-\number\month-\number\day,
          \thehours.\ifnum\theminutes<10{0}\fi\theminutes }
} \makeatletter\def\thefootnote{\@arabic\c@footnote}\makeatother

\begin{abstract} 
	A class $ \mathcal{F} $ consisting of analytic functions $ f(z)=\sum_{n=0}^{\infty}a_nz^n $ in the unit disc $ \mathbb{D}=\{z\in\mathbb{C}:|z|<1\} $ satisfies a Bohr phenomenon if there exists an $ r_f>0 $ such that
	\begin{equation*}
		I_f(r):=\sum_{n=1}^{\infty}|a_n|r^n\leq{d}\left(f(0),\partial \mathbb{D}\right)
	\end{equation*}
for every function $ f\in\mathcal{F} $, and $ |z|=r\leq r_f $. The largest radius $ r_f $ is the Bohr radius and the inequality $ I_f(r)\leq{d}\left(f(0),\partial \mathbb{D}\right) $ is Bohr inequality for the class $ \mathcal{F} $, where `$ d $' is the Euclidean distance. If there exists a positive real number $ r_0 $ such that $ I_f(r)\leq {d}\left(f(0),\partial \mathbb{D}\right) $ holds for every element of the class $ \mathcal{F} $ for $ 0\leq r<r_0 $ and fails when $ r>r_0 $, then we say that $ r_0 $ is sharp bound for the inequality w.r.t. the class $ \mathcal{F} $. In this paper, we prove sharp refinement of the Bohr-Rogosinski inequality for certain classes of harmonic mappings.
\end{abstract}

\maketitle
\pagestyle{myheadings}
\markboth{Molla Basir Ahamed}{The sharp refined Bohr-Rogosinski inequalities for certain classes of harmonic mappings}

\section{Introduction}
The origin of Bohr phenomenon lies in the seminal work by Harald Bohr
\cite{Bohr-1914} in $ 1914 $ for the analytic functions of the form $ \sum_{n=0}^{\infty}a_nz^n $ defined on the unit disk $ \mathbb{D}:=\{z\in\mathbb{C} : |z|<1\} $ with $ |f(z)|<1 $. This classical result found an application
to the characterization problem of Banach algebras satisfying the von Neumann inequality (see \cite{Dixon & BLMS & 1995}), the
Bohr inequality has attracted many researchers’ attention in the function theory. Study of Bohr phenomenon for different classes of functions with various settings becomes a subject of great interests during past several years and an extensive research work has been done by many authors (see e.g., \cite{Alkhaleefah-PAMS-2019,Paulsen-PLMS-2002,Paulsen-PAMS-2004,Paulsen-BLMS-2006,Ismagilov-2020-JMAA,Kayumov-CRACAD-2018,Kayumov-JMAA-2021,Kayumov-MJM-2022,Kay & Pon & AASFM & 2019,Kumar-CVEE-2022,Lata-Singh-PAMS-2022,Liu-JMAA-2021,Liu-Ponnusamy-MN,Liu-Ponnusamy-PAMS-2021,Ponnusamy-HJM-2021} and references therein). The Bohr phenomenon for Hardy space functions-both in single and several variables-along with some Schwarz-Pick
type estimates are established in \cite{Bene-2004}. The Bohr-type inequalities for holomorphic mappings with a lacunary series in several complex variables are obtained recently in \cite{Lin-Liu-Ponn-AMS-2023}. However, Bohr phenomenon for the classes of harmonic mappings was initiated first in \cite{Abu-CVEE-2010} and was investigated
in \cite{Kay & Pon & Sha & MN & 2018} and subsequently by a number of authors, e.g. \cite{Aha-CMFT-2022,Aha-Allu-BMMSS-2022,Aha-Allu-RMJ-2022,Ahamed-AMP-2021,Ahamed-CVEE-2021,Huang-Liu-Ponnusamy-MJM-2021,Kayumov-2018-JMAA,Liu-Ponnusamy-BMMS-2019,Liu-Ponn-Wang-RACSAM-2020}. For different aspects of Bohr phenomenon including multidimensional Bohr inequality, the readers are referred to the articles \cite{Ahamed-AASFM-2022,Aizn-PAMS-1999,Aizenberg-SM-2005,Aigenber-CMFT-2009,Ali-JMAA-2017,Aizenberg-AMP-2012,Allu-CMB-2022,Allu-JMAA-2021,Kumar-PAMS-2022,Ponnusamy-JMAA-2022,Bhowmik-2018,Blasco-2010,Boas-Khavinson-1997,Das-CMD-2022,Dixon & BLMS & 1995,Defant-IJF-2006,Defant-AM-2012,Evdoridis-IM-2018,Galicer-TAMS-2021} and references therein. In this paper, we are mainly interested to study Bohr phenomenon with suitable settings in order to establish certain harmonic analogue of some Bohr inequality valid for analytic functions. The recent survey article \cite{Ali-Abu-Ponn-2016,Ponnusamy-Vijayakumar} and references therein may be good sources for this topic.\vspace{1.2mm}

For a continuously differentiable complex-valued mapping $ f(z)=u(z)+iv(z) $, $ z=x+iy $, we use the common notions for its formal derivatives:
\begin{align*}
	f_{z}=\frac{1}{2}\left(f_x-if_y\right)\;\; \mbox{and}\;\; f_{\bar{z}}=\frac{1}{2}\left(f_x+if_y\right).
\end{align*}
We say that $ f $ is a harmonic mapping in a simply connected domain $ \Omega $ if $ f $ is twice continuously differentiable and satisfies the Laplacian equation $ \Delta f=4f_{z\bar{z}}=0 $ in $ \Omega $, where $ \Delta $ is the complex Laplacian operator defined by
$ \Delta={\partial^2}/{\partial x^2}+{\partial^2}/{\partial y^2}. $\vspace{1,2mm}
\par Methods of harmonic mappings have been applied to study and solve the fluid flow problems (see \cite{Aleman-2012,Constantin-2017}). For example, in 2012, Aleman and Constantin \cite{Aleman-2012} established a connection between harmonic mappings and ideal fluid flows. In fact, Aleman and Constantin have developed ingenious technique to solve the incompressible two dimensional Euler equations in terms of univalent harmonic mappings (see \cite{Constantin-2017} for details).\vspace{1.2mm}

 Let $ \mathcal{H}(\Omega) $ be the class of complex-valued functions harmonic in $ \Omega $. It is well-known that functions $ f $ in the class $ \mathcal{H}(\Omega) $ has the following representation $ f=h+\overline{g} $, where $ h $ and $ g $ both are analytic functions in $ \Omega $. The famous Lewy's theorem \cite{Lew-BAMS-1936} in $ 1936 $ states that a harmonic mapping $ f=h+\overline{g} $ is locally univalent on $ \Omega $ if, and only if, the determinant $ |J_f(z)| $ of its Jacobian matrix $ J_f(z) $ does not vanish on $ \Omega $, where
\begin{equation*}
	|J_f(z)|:=|f_{z}(z)|^2-|f_{\bar{z}}(z)|^2=|h^{\prime}(z)|^2-|g^{\prime}(z)|^2\neq 0.
\end{equation*}
In view of this result, a locally univalent harmonic mapping  is sense-preserving if $ |J_f(z)|>0 $ and sense-reversing if $|J_{f}(z)|<0$ in $\Omega$. For detailed information about the harmonic mappings, we refer the reader to \cite{Clunie-AASF-1984,Duren-2004}. In \cite{Kay & Pon & Sha & MN & 2018}, Kayumov \emph{et al.} first established the harmonic extension of the classical Bohr theorem, since then investigating on the Bohr-type inequalities for certain class of harmonic mappings becomes an interesting topic of research in geometric function theory.\vspace{1.2mm}
\par Let $ \mathcal{A} $ denote the set of all analytic functions of the form $ f(z)=\sum_{n=0}^{\infty}a_nz^n $ defined on $ \mathbb{D} $ and we define the class $ \mathcal{B}=\{f\in\mathcal{A} : |f(z)|\leq 1\; \mbox{in}\; \mathbb{D}\} $.  Let us first recall the theorem of Bohr \cite{Bohr-1914} in $ 1914 $, which inspired a lot in the recent years.

\begin{thm}\cite{Bohr-1914}\label{th-1.1}
	If $ f(z)=\sum_{n=0}^{\infty}a_nz^n\in\mathcal{B} $, then 
	\begin{equation}\label{e-1.1}
		M_f(r):=\sum_{n=0}^{\infty}|a_n|r^n\leq 1 \;\; \mbox{for}\;\; |z|=r\leq\frac{1}{3}.
	\end{equation}
\end{thm}
Initially, Bohr showed the inequality \eqref{e-1.1} for $|z|\leq1/6$,  but later  M. Riesz, I. Schur and F. Wiener subsequently improved the inequality \eqref{e-1.1} for $|z|\leq1/3$ and showed that the constant $1/3$ is best possible. It is quite natural that the constant $1/3$ and the inequality \eqref{e-1.1} are called respectively, the Bohr radius and Bohr inequality for the class $\mathcal{B}$. Moreover, for:
\begin{equation*}
	f_a(z)=\frac{a-z}{1-az},\;\; a\in [0,1)
\end{equation*} it follows easily that $ M_{f_a}(r)>1 $ if, and only if, $ r>1/(1+2a) $, which shows that $ 1/3 $ is best possible as the limiting case $ a\rightarrow 1 $ suggests.
 \vspace{1.5mm}
 
 Bohr phenomenon can be studied in view of Euclidean distance and in this paper, we study the same for certain classes of harmonic mappings. Before we go into details, we recall here the following concepts. Let $ f $ and $ g $ be two analytic functions in the unit disc $ \mathbb{D} $. We say that $ g $ is subordinate to $ f $ if there is a function $ \varphi $, analytic in $ \mathbb{D} $, $ \varphi(\mathbb{D})\subset \mathbb{D} $ and $ \varphi(0)=0 $ so that $ g=f\circ \varphi $. In particular, when the function $ f $ is univalent, $ g $ is subordinate to $ f $ when $ g(\mathbb{D}\subset f(\mathbb{D})) $ and $ g(0)=f(0) $ (see \cite[p. 190]{Duren-1983}). Consequently, when $ g $ is subordinate to $ f $, $ |g^{\prime}(0)|\leq |f^{\prime}(0)| $. The class of all function $ g $ subordinate to a fixed function $ f $ is denoted by $ \mathcal{S}(f) $ and $ f(\mathbb{D})=\Omega $.
 \begin{defi}\cite{Abu-CVEE-2010}
 We say that $ \mathcal{S}(f) $ has Bohr phenomenon if for any $ g(z)=\sum_{n=0}^{\infty}b_nz^n\in\mathcal{S}(f) $ and $ f(z)=\sum_{n=0}^{\infty}a_nz^n $ there is a $ \rho^*_0 $, $ 0<\rho^*_0\leq 1 $ so that $ \sum_{n=1}^{\infty}|b_nz^n|\leq d(f(0), \partial \Omega)  $,  for $ |z|<\rho^*_0 $. Notice that $  d(f(0), \partial \Omega) $ denote the Euclidean distance between $ f(0) $ and the boundary of a domain $ \Omega $, $ \partial\Omega $. In particular, when $ \Omega=\mathbb{D} $, $  d(f(0), \partial \Omega)=1-|f(0)| $ and in this case $ \sum_{n=1}^{\infty}|a_nz^n|\leq d(f(0), \partial \Omega) $ reduces to $ \sum_{n=0}^{\infty}|a_nz^n|\leq 1 $.
 \end{defi}
\par Equation \eqref{e-1.1} can be written as
\begin{equation}\label{e-1.2}
	d\left(\sum_{n=0}^{\infty}|a_nz^n|,|a_0|\right)=\sum_{n=1}^{\infty}|a_nz^n|\leq 1-|f(0)|=d(f(0),\partial (\mathbb{D})),
\end{equation}
where $ d $ is the Euclidean distance. More generally, a class $ \mathcal{F} $ of analytic functions $ f(z)=\sum_{n=0}^{\infty}a_nz^n $ mapping $ \mathbb{D} $ into a domain $ \Omega $ is said to satisfy a Bohr phenomenon if an inequality of type \eqref{e-1.2} holds uniformly in $ |z|\leq \rho_0 $, where $ 0<\rho_0\leq 1 $ for functions in the class $ \mathcal{F} $. Similar definition makes sense for harmonic functions (see \cite{Kay & Pon & Sha & MN & 2018}).\vspace{1.2mm}

Abu-Muhanna \cite{Abu-CVEE-2010} have established the following result for subordination $ S(f) $ when $ f $ is univalent.
\begin{thm}\cite{Abu-CVEE-2010}
	If $ g(z)=\sum_{n=0}^{\infty}b_nz^n\in S(f) $ and $ f(z)=\sum_{n=0}^{\infty}a_nz^n $ is univalent, then 
	\begin{align*}
		\sum_{n=0}^{\infty}|b_nz^n|^n\leq d(f(0), \partial\Omega)
	\end{align*}
for $ |z|=\rho_0^*\leq 3-\sqrt{8}\approx 0.17157, $ where $ \rho_0^* $ is sharp for Koebe function $ f(z)=z/(1-z)^2 $.
\end{thm}

\par Let $ \mathcal{H} $ be the class of all complex-valued harmonic functions $ f=h+\bar{g} $ defined on the unit disk $ \mathbb{D} $, where $ h $ and $ g $ are analytic in $ \mathbb{D} $ with the normalization $ h(0)=h^{\prime}(0)-1=0 $ and $ g(0)=0 $. Let $ \mathcal{H}_0 $ be defined by $ 	\mathcal{H}_0=\{f=h+\bar{g}\in\mathcal{H} : g^{\prime}(0)=0\}. $ Therefore, each $f=h+\overline{g}\in \mathcal{H}_{0}$ has the following representation 
\begin{equation}\label{e-1.3}
	f(z)=h(z)+\overline{g(z)}=\sum_{n=1}^{\infty}a_nz^n+\overline{\sum_{n=1}^{\infty}b_nz^n}=z+\sum_{n=2}^{\infty}a_nz^n+\overline{\sum_{n=2}^{\infty}b_nz^n},
\end{equation}
where $ a_1 = 1 $ and $ b_1 = 0 $, since $ a_1 $ and $ b_1 $ have been appeared in later results and corresponding proofs.\vspace{1.2mm}

Let us recall the Bohr radius for the class of harmonic mappings.
\begin{defi}
	Let $ f\in\mathcal{H}_0 $ be given by \eqref{e-1.3}. Then the Bohr phenomenon is to find a constant $ R^*\in (0, 1] $ such that the inequality $ r+\sum_{n=2}^{\infty}\left(|a_n|+|b_n|\right)r^n\leq d\left(f(0), \partial\Omega\right) $ holds for $ |z|=r\leq R^* $, where  $ d\left(f(0), \partial\Omega\right) $ is the Euclidean distance between $ f(0) $ and the boundary of $ \Omega:=f(\mathbb{D}) $. The largest such radius $ R^* $ is called the Bohr radius for the class $ \mathcal{H}_0 $.
\end{defi}
Based on the notion of Rogosinski’s inequality and Rogosinski’s radius investigated in \cite{Rogosinski-1923}, in $ 2017 $, Kayumov and Ponnusamy \cite{Kayumov-2017} introduced and obtained the following Bohr-Rogosinski inequality and Bohr-Rogosinski radius for the class $ \mathcal{B} $.
\begin{thm}\cite{Kayumov-2017}\label{th-1.2}
	Suppose that  $ f\in\mathcal{B} $ with $f(z)=\sum_{n=0}^{\infty} a_{n}z^{n}$. Then 
\begin{equation}\label{e-11.44}
	 A_f(z):=|f(z)|+\sum_{n=N}^{\infty}|a_n|r^n\leq 1\;\;\mbox{for}\;\; r\leq R_N,
\end{equation}
where $ R_N $ is the positive root of the equation $ 2(1+r)r^N-(1-r)^2=0 $. The radius $ R_N $ is best possible. Moreover, 
\begin{equation}\label{e-11.55}
	B_f(z):=|f(z)|^2+\sum_{n=N}^{\infty}|a_n|r^n\leq 1\;\;\mbox{for}\;\; r\leq R^{\prime}_N,
\end{equation}
where $ R^{\prime}_N $ is the positive root of the equation $ (1+r)r^N-(1-r)^2=0 $. The radius $ R^{\prime}_N $ is best possible.
\end{thm}
For an extension of this results, we refer to the recent article by Ponnusamy and Vijayakumar \cite{Ponnusamy-Vijayakumar}. In comparison of $ \sum_{n=0}^{\infty}|a_n|r^n $ with another functional often considered in function theory, namely $ \sum_{n=0}^{\infty}|a_n|^2r^{2n} $ which is abbreviated as $ ||f||^2_r $. As refinement of the classical Bohr inequality, for $ j=1, 2 $, it can be defined
\begin{align*}
	\mathcal{A}_{j, f}(r):= |a_0|^j+\sum_{n=1}^{\infty}|a_n|r^n+\left(\frac{1}{1+|a_0|}+\frac{r}{1-r}\right)||f_0||^2_r,
\end{align*}
where $ f_0(z):=f(z)-a_0 $. In $ 2020 $, Ponnusamy \emph{et al.} \cite{Ponnusamy-RM-2020} proved the following result as a refinement of the classical Bohr inequality.
\begin{thm}\cite{Ponnusamy-RM-2020}\label{th-11.33}
Suppose that $ f\in\mathcal{B} $ with $ f(z)=\sum_{n=0}^{\infty}a_nz^n $ and $ f_0(z):=f(z)-a_0 $. Then $ \mathcal{A}_{1, f}(r)\leq 1 $ and the numbers $ 1/(1+|a_0|) $ and $ 1/(2+|a_0|) $ cannot be improved. Further, $ \mathcal{A}_{2, f}(r)\leq 1 $ and the numbers $ 1/(1+|a_0|) $ and $ 1/2 $ cannot be improved.
\end{thm}
In what follows, $ \lfloor x \rfloor $ denotes the largest integer no more than $ x $, where $ x  $is a real number. Recently, Liu \emph{et al.} \cite{Liu-Liu-Ponnusamy-BSM-2021} obtained the following refined version of Bohr-Rogosinski inequality.
\begin{thm}\cite{Liu-Liu-Ponnusamy-BSM-2021}\label{th-1.3}
	Suppose that $ f\in\mathcal{B} $ and $ f(z)=\sum_{n=0}^{\infty}a_nz^n $. For $ N\in\mathbb{N} $, let $ t=\lfloor (N-1)/2 \rfloor $. Then 
	\begin{align}\label{e-11.88}
		|f(z)|&+\sum_{n=N}^{\infty}|a_n|r^n+sgn(t)\sum_{n=1}^{t}|a_n|^2\frac{r^N}{1-r}\\&\nonumber+\left(\frac{1}{1+|a_0|}+\frac{r}{1-r}\right)\sum_{n=t+1}^{\infty}|a_n|^2r^{2n}\leq 1
	\end{align} 
for $ |z|=r\leq R_N $, where $ R_N $ is as in Theorem \ref{th-1.2}. The radius $ R_N $ is best possible.
\begin{align}\label{e-11.99}
	|f(z)|^2&+\sum_{n=N}^{\infty}|a_n|r^n+sgn(t)\sum_{n=1}^{t}|a_n|^2\frac{r^N}{1-r}\\&\nonumber+\left(\frac{1}{1+|a_0|}+\frac{r}{1-r}\right)\sum_{n=t+1}^{\infty}|a_n|^2r^{2n}\leq 1
\end{align}
for $ |z|=r\leq R^{\prime}_N $, where $ R^{\prime}_N $ is as in Theorem \ref{th-1.2}. The radius $ R^{\prime}_N $ is best possible.
\end{thm}
For $ N=1 $, it is easy to see that $ R_1=\sqrt{5}-2 $ and $ R^{\prime}_1=1/3 $. For $ j=1, 2 $, we define here some notations
\begin{align*}
	\mathcal{B}_{j, f}(z, r):=|f(z)|^j+\sum_{n=1}^{\infty}|a_n|r^n+\left(\frac{1}{1+|a_0|}+\frac{r}{1-r}\right)\sum_{n=1}^{\infty}|a_n|^2r^{2n}
\end{align*}In the context of Theorems \ref{th-1.2} and \ref{th-1.3}, recently, Liu \emph{et al.} \cite{Liu-Liu-Ponnusamy-BSM-2021} obtained the following result showing that the two constants can be improved for any individual function in $ \mathcal{B}. $
\begin{thm}\cite{Liu-Liu-Ponnusamy-BSM-2021}\label{th-1.4}
	Suppose that $ f\in\mathcal{B} $ and $ f(z)=\sum_{n=0}^{\infty}a_nz^n $. Then $ \mathcal{B}_{1, f}(z, r)\leq 1 $ for $ |z|=r\leq r_{a_0}=2/(3+|a_0|+\sqrt{5}(1+|a_0|))$. The radius $ r_{a_0} $ is best possible and $ r_{a_0}>\sqrt{5}-2 $. Moreover, $ \mathcal{B}_{2, f}(z, r)\leq 1 $ for $ |z|=r\leq r^{\prime}_{a_0} $, where $ r^{\prime}_{a_0} $ is the unique positive root of the equation
	\begin{align*}
		\left(1-|a_0|^3\right)r^3-(1+2|a_0|)r^2-2r+1=0.
	\end{align*}
 The radius $ r^{\prime}_{a_0} $ is best possible. Further, $ 1/3<r^{\prime}_{a_0}<1/(2+|a_0|) $.
\end{thm}
For recent developments on the Bohr-Rogosinski inequalities, we refer to the articles \cite{Aha-CMFT-2022,Alkhaleefah-LJM-2021,Allu-Arora-JMAA-2022,Allu-Halder-Pal-BDS-2022,Das-JMAA-2022}. We see that the quantities $ 1/(1+|a_0|)+r/(1-r) $ and $ \sum_{n=1}^{\infty}|a_n|^2r^{2n} $ for analytic functions in $ \mathcal{B} $ are analogous to $ 1/(1+|a_0|+|b_0|)+r/(1-r)=1+r/(1-r) $ (as $ a_0=0=b_0 $) and $ \sum_{n=2}^{\infty}(|a_n|+|b_n|)^2r^{2n} $, respectively, for harmonic functions given in \eqref{e-1.3}. \vspace{1.2mm}
The above discussions motivate us to give certain harmonic analogues of the refined Bohr inequalities. Hence, a natural inquisition is the following.
\begin{prob}\label{p-1}
	Can we establish harmonic analogue of Theorems \ref{th-1.3} and \ref{th-1.4} for certain classes of harmonic mappings?
\end{prob}

Motivated from the paper \cite{Ahamed-CVEE-2021}, our aim in this paper is to find the solution of Problem \ref{p-1} in order to establish harmonic analogue of Theorems \ref{th-1.3} and \ref{th-1.4} for certain classes of harmonic mappings discussed in the subsections of Section 2. The coefficient bounds and the growth theorems for functions in each class are stated, and the main results and their proofs are discussed in details in each subsection. \vspace{1.2mm}

\section{Main results}
Before stating the main results, we recall here the definition of dilogarithm $ {\rm Li}_2(z) $ which is defined by the power series
\begin{equation*}
	{\rm Li}_2(z)=\sum_{n=1}^{\infty}\frac{z^n}{n^2}\;\; \mbox{for}\;\; |z|<1.
\end{equation*}
The following equality of dilogarithm holds
\begin{equation*}
	{\rm Li}_2(r)+{\rm Li}_2(1-r)=\frac{\pi^2}{6}-\log r\log (1-r).
\end{equation*}
Moreover, 
\begin{align*}
	{\rm Li}_2(1-r^2)\rightarrow 0\;\; \mbox{and}\;\; \frac{\log(1/r)}{1-r}\rightarrow 1\;\; \mbox{as}\;\; r\rightarrow 1.
\end{align*} 
In particular, the analytic continuation of the dilogarithm is given by
\begin{equation*}
	{\rm Li}_2(z)=-\int_{0}^{z}\log (1-u)\frac{du}{u}\;\; \mbox{for}\;\; z\in\mathbb{C}\setminus [1,\infty).
\end{equation*}
The Euclidean distance between $f(0)$ and the boundary of $f(\mathbb{D})$ is given by 
\begin{equation} \label{e-2.1}
	d(f(0), \partial f(\mathbb{D}))= \liminf \limits_{|z|=r\rightarrow 1} |f(z)-f(0)|.
\end{equation}
\subsection{Refined Bohr-type inequality for the class $ \mathcal{P}^{0}_{\mathcal{H}}(\alpha) $} Motivated by the class $\mathcal{P}_{\mathcal{H}}^{0}$ in \cite{Ponnusamy-CVEE-2013}, where $$\mathcal{P}_{\mathcal{H}}^{0}=\{f=h+\bar{g} \in \mathcal{H} : \real h^{\prime}(z)>|g^{\prime}(z)| \;\; \mbox{with} \;\; g^{\prime}(0)=0 \;\; \mbox{for}\;\; z \in \mathbb{D}\},
$$ in $ 2013 $, Li and Ponnusamy \cite{Li-Ponnusamy-NA-2013} have studied the growth estimates and sharp coefficients bounds of the function $ f $ in the class $\mathcal{P}_{\mathcal{H}}^{0}(\alpha)$ which is defined by
$$\mathcal{P}_{\mathcal{H}}^{0}(\alpha)=\{f=h+\overline{g} \in \mathcal{H} : \real (h^{\prime}(z)-\alpha)>|g^{\prime}(z)|,\; 0\leq\alpha<1,\; g^{\prime}(0)=0\; \mbox{for}\;  z \in \mathbb{D}\}.
$$ The Bohr phenomenon has been studied recently for the class $ \mathcal{P}_{\mathcal{H}}^{0}(\alpha) $ in the paper \cite{Ahamed-AMP-2021}. To study the refined Bohr inequalities for functions in $  \mathcal{P}^{0}_{\mathcal{H}}(\alpha) $, the key ingredient of our investigation are the following coefficient bounds and growth estimates for functions in the class $\mathcal{P}_{H}^{0}(\alpha)$ which were proved by Li and Ponnusamy \cite{Li-Ponnusamy-NA-2013}, and Allu and Halder \cite{Allu-BSM-2021}, respectively.
\begin{lem} \label{lem-2.3}  \cite{Li-Ponnusamy-NA-2013}
	Let $f \in \mathcal{P}^{0}_{\mathcal{H}}(\alpha) $ and be given by \eqref{e-1.3}. Then for any $n \geq 2$, 
	\begin{enumerate}
		\item[(i)] $\displaystyle |a_n| + |b_n|\leq \frac {2(1-\alpha)}{n}; $\\[1mm]
		
		\item[(ii)] $\displaystyle ||a_n| - |b_n||\leq \frac {2(1-\alpha)}{n};$\\[1mm]
		
		\item[(iii)] $\displaystyle |a_n|\leq \frac {2(1-\alpha)}{n}.$
	\end{enumerate}
	All the inequalities  are sharp, with extremal function $f(z)=(1-\alpha)(-z-2\,\;\log(1-z))+\alpha z$.
\end{lem}
\begin{lem} \cite{Allu-BSM-2021}\label{lem-2.4}
	Let $f=h+\overline{g} \in \mathcal{P}^{0}_{\mathcal{H}}(\alpha)$ with $0\leq \alpha <1$. Then 
	\begin{equation*}
		|z|+ \sum\limits_{n=2}^{\infty}  \dfrac{2(1-\alpha)(-1)^{n-1}}{n} |z|^{n} \leq |f(z)| \leq |z|+ \sum\limits_{n=2}^{\infty}  \dfrac{2(1-\alpha)}{n} |z|^{n}.
	\end{equation*} 
	Both  inequalities are sharp.
\end{lem}

Since $f(0)=0$, then in view of Lemma \ref{lem-2.4} and \eqref{e-2.1}, a simple computation shows that
\begin{equation} \label{e-2.3}
	d(f(0), \partial f(\mathbb{D})) \geq 1+\sum\limits_{n=2}^{\infty}  2(1-\alpha) \dfrac{(-1)^{n-1}}{n}=1+2(1-\alpha)(\ln 2-1).
\end{equation}
Before, we stating the main results of the paper, we introduce here the notation:
\begin{align*}
	S^f_{\mu, \lambda, m, N}(r):&= |f(z)|^m+\sum_{n=N}^{\infty}\left(|a_n|+|b_n|\right)r^n+\mu\; sgn(t)\sum_{n=1}^{t}\left(|a_n|+|b_n|\right)^2\frac{r^N}{1-r}\\&\nonumber\quad+\lambda\left(1+\frac{r}{1-r}\right)\sum_{n=t+1}^{\infty}\left(|a_n|+|b_n|\right)^2r^{2n}.
\end{align*}
Let $ N $ be a positive integer. We notice that if $ N=1, 2 $, then $ t=\lfloor (N-1)/2 \rfloor=0 $, and hence $ sgn(t)=0 $, and in case when $ N=3, 4 $, then we see that $ t=\lfloor (N-1)/2 \rfloor=1 $, and when $ N\geq 5 $, then $ t\geq 2 $ with $ sgn(t)=1 $. In view of this observations, to serve our purpose, in this paper, in all the main results, we will consider $ N\geq 5 $. The possible situation for the cases $ N=1, 2, 3, 4 $, we give corollary of the corresponding main result.\vspace{1.2mm}

We now state our first main result which is a sharp improved Bohr inequality for the class $ \mathcal{P}^{0}_{\mathcal{H}}(\alpha) $.
\begin{thm}\label{th-2.6}
	Let $ f\in \mathcal{P}^{0}_{\mathcal{H}}(\alpha) $ be given by \eqref{e-1.3} and $ 0\leq \alpha <1 $ and $ N $ be a positive integer. For $ N\geq 5 $, $ t=\lfloor (N-1)/2 \rfloor $ and $ \mu,\; \lambda\in\mathbb{R}_{\geq 0}:=\{x\in\mathbb{R} : x\geq 0\} $, we have $ S^f_{\mu, \lambda, m,  N}(r) \leq {d}\left(f(0),\partial \mathbb{D}\right) $ 
	for $ |z|=r\leq R^{m, N, t}_{1,\mu, \lambda}(\alpha) $, where $ R^{m, N, t}_{1,\mu, \lambda}(\alpha) $ is the unique root of the equation $ 	\Phi^{N,m,\alpha}_{1,\mu, \lambda,t}(r)=0 $ in $ (0,1) $, where
\begin{align}
	\label{e-2.5}
	\Phi^{N,m,\alpha}_{1,\mu, \lambda,t}(r):&= \left(F_{\alpha}(r)\right)^m-2(1-\alpha)\left(\ln(1-r)+\sum_{n=1}^{N-1}\frac{r^n}{n}\right)+\mu G^N_{\alpha, t}(r)\\&\nonumber\quad-\lambda H_{\alpha,t}(r)-1-2(1-\alpha)\left(\ln 2-1\right)
\end{align}
and
\[
\begin{cases}
	F_{\alpha}(r):=r-2(1-\alpha)\left(r+\ln(1-r)\right)\vspace{1.4mm}\\
	G^N_{\alpha, t}(r):=\displaystyle\frac{4(1-\alpha)^2r^N}{1-r}sgn(t)\sum_{n=1}^{t}\frac{1}{n^2}\vspace{1.4mm}\\
	H_{\alpha, t}(r):=\displaystyle 4(1-\alpha)^2\left(1+\frac{r}{1-r}\right)\left(\sum_{n=1}^{t}\frac{r^{2n}}{n^2}-{\rm Li_2\left(r^2\right)}\right).
\end{cases}
\]
The constant $ R^{m, N, t}_{1,\mu, \lambda}(\alpha) $ is best possible.
\end{thm}
\begin{rem} The following observations are clear. For the triplets $ (0, 0, 1)$, $(0, 0, 2)$, $(1, 1, 1)$, $(1, 1, 2) $ corresponding to $ (\lambda, \mu, m) $, the inequality $ S^f_{\mu, \lambda, m, N}(r) \leq {d}\left(f(0),\partial \mathbb{D}\right) $ is harmonic analog of \eqref{e-11.44}, \eqref{e-11.55}, \eqref{e-11.88} and \eqref{e-11.99}, respectively, for the class $ \mathcal{P}^{0}_{\mathcal{H}}(\alpha) $. 
\end{rem}
We define the following notations:
\[
\begin{cases}
	\mathcal{J}^{\alpha}_1(r):=r^2+4(1-\alpha)^2[{\rm Li_2\left(r^2\right)}-r^2]\\
	\mathcal{J}^{\alpha}_2(m, r):=(F_{\alpha}(r))^m-1-2(1-\alpha)(\ln 2-1)\\
	\mathcal{J}_3(r):=r+\ln(1-r).
\end{cases}
\]
As a corollary of Theorem \ref{th-2.6}, we obtain the following result in which the cases for $ N=1, 2, 3, 4 $ are discussed. The proof can be carried in the line of the proof of Theorem \ref{th-2.6}, hence we omit the details.
\begin{cor} Let $ f\in \mathcal{P}^{0}_{\mathcal{H}}(\alpha) $ be given by \eqref{e-1.3} and $ 0\leq\alpha<1 $, $ \mu,\; \lambda\in\mathbb{R}_{\geq 0} $.
	\begin{enumerate}
		\item[(i)] If $ N=1 $, then $ S^f_{\mu, \lambda, m, 1}(r)\leq {d}\left(f(0),\partial \mathbb{D}\right) $
		for $ |z|=r\leq R^{m, 1, 0}_{1,\mu, \lambda}(\alpha) $, where $ R^{m, 1, 0}_{1,\mu, \lambda}(\alpha) $ is the unique root in $ (0,1) $ of the equation 
		\begin{align*}
			r-2(1-\alpha)\mathcal{J}_3(r)+\mathcal{J}^{\alpha}_2(m, r)+\lambda\left(1+\frac{r}{1-r}\right)\mathcal{J}^{\alpha}_1(r)=0.
		\end{align*}
		\item[(ii)] If $ N=2 $, then $ S^f_{\mu, \lambda, m, 2}(r)\leq {d}\left(f(0),\partial \mathbb{D}\right) $
		for $ |z|=r\leq R^{m, 2, 0}_{1,\mu, \lambda}(\alpha) $, where $ R^{m, 2, 0}_{1,\mu, \lambda}(\alpha) $ is the unique root in $ (0,1) $ of the equation 
		\begin{align*}
			-2(1-\alpha)\mathcal{J}_3(r)+\mathcal{J}^{\alpha}_2(m, r)+\lambda\left(1+\frac{r}{1-r}\right)\mathcal{J}^{\alpha}_1(r)=0.
		\end{align*}
		\item[(iii)] If $ N=3 $, then $ S^f_{\mu, \lambda, m, 3}(r)\leq {d}\left(f(0),\partial \mathbb{D}\right) $
		for $ |z|=r\leq R^{m, 3, 1}_{1,\mu, \lambda}(\alpha) $, where $ R^{m, 3, 1}_{1,\mu, \lambda}(\alpha) $ is the unique root in $ (0,1) $ of the equation 
		\begin{align*}
			-2(1-\alpha)\left(\mathcal{J}_3(r)+\frac{r^2}{2}\right)+\mathcal{J}^{\alpha}_2(m, r)+\mu\frac{r^3}{1-r}+\lambda\left(1+\frac{r}{1-r}\right)\left(\mathcal{J}^{\alpha}_1(r)-r^2\right)=0.
		\end{align*}
		\item[(iv)] If $ N=4 $, then $ S^f_{\mu, \lambda, m, 4}(r)\leq {d}\left(f(0),\partial \mathbb{D}\right) $
		for $ |z|=r\leq R^{m, 4, 1}_{1,\mu, \lambda}(\alpha) $, where $ R^{m, 4, 1}_{1,\mu, \lambda}(\alpha) $ is the unique root in $ (0,1) $ of the equation 
		\begin{align*}
		-2(1-\alpha)\left(\mathcal{J}_3(r)+\frac{r^2}{2}+\frac{r^3}{3}\right)+\mathcal{J}^{\alpha}_2(m, r)+\mu\frac{r^4}{1-r}+\lambda\left(1+\frac{r}{1-r}\right)\left(\mathcal{J}^{\alpha}_1(r)-r^2\right)=0.
		\end{align*}
	\end{enumerate}
	The constants $ R^{m, 1, 0}_{1,\mu, \lambda}(\alpha) $, $ R^{m, 2, 0}_{1,\mu, \lambda}(\alpha) $, $ R^{m, 3, 1}_{1,\mu, \lambda}(\alpha) $ and $ R^{m, 4, 1}_{1,\mu, \lambda}(\alpha) $ are best possible.
\end{cor}
\begin{proof}[\bf Proof of Theorem \ref{th-2.6}]
	In view of Lemma \ref{lem-2.4}, we have 
	\begin{align}\label{e-2.6}
		|f(z)|\leq r+2(1-\alpha)\sum_{n=2}^{\infty}\frac{r^n}{n}=r-2(1-\alpha)\left(r+\ln(1-r)\right)=F_{\alpha}(r)
	\end{align}
Using the coefficients bound in Lemma \ref{lem-2.3}, a straight forward computation shows that 
\begin{align*}
	\sum_{n=N}^{\infty}\left(|a_n|+|b_n|\right)r^n\leq \sum_{n=N}^{\infty}\frac{2(1-\alpha)r^n}{n}=-2(1-\alpha)\left(\ln(1-r)+\sum_{n=1}^{N-1}\frac{r^n}{n}\right)\nonumber
\end{align*} 
\begin{align}\label{e-2.7}
	\sum_{n=t+1}^{\infty}\left(|a_n|+|b_n|\right)^2r^{2n}\leq -4(1-\alpha)^2\left(\sum_{n=1}^{t}\frac{r^{2n}}{n^2}-{\rm Li_2\left(r^2\right)}\right)
\end{align}
and 
\begin{align}\label{e-2.8}
	sgn(t)\sum_{n=1}^{t}\left(|a_n|+|b_n|\right)^2\frac{r^N}{1-r}\leq \frac{4(1-\alpha)^2sgn(t)r^N}{1-r}\sum_{n=1}^{t}\frac{1}{n^2}=: G^N_{\alpha, t}(r).
\end{align}
Therefore, a simple computation using \eqref{e-2.6}, \eqref{e-2.7} and \eqref{e-2.8} shows that 
\begin{align}\label{e-2.9}
	S^f_{\mu, \lambda, m, N}(r)&\leq \left(F_{\alpha}(r)\right)^m-2(1-\alpha)\left(\ln(1-r)+\sum_{n=1}^{N-1}\frac{r^n}{n}\right)+\mu G^N_{\alpha, t}(r)-\lambda H_{\alpha,t}(r)\\&\leq 1+2(1-\alpha)\left(\ln 2-1\right)\nonumber
\end{align}
for $ |z|=r\leq R^{m, N, t}_{1,\mu, \lambda}(\alpha) $, where $ R^{m, N, t}_{1,\mu, \lambda}(\alpha) $ is the smallest root of the equation $ 	\Phi^{N,m,\alpha}_{1,\mu, \lambda,t}(r)=0 $ in $ (0,1) $ and $ \Phi^{N,m,\alpha}_{1,\mu, \lambda,t} : [0, 1]\rightarrow\mathbb{R} $ is defined in \eqref{e-2.5}.
By a routine computation, it can be easily shown that
\begin{align}\label{e-2.11}
	\frac{d}{dr}\left(\Phi^{N,m,\alpha}_{1,\mu, \lambda,t}(r)\right)>0\; \mbox{for}\; r\in (0, 1),
\end{align}
 and hence $ \Phi^{N,m,\alpha}_{1,\mu, \lambda,t}(r) $ is an increasing function of $ r $ in the interval $ (0, 1) $. We see that $ \Phi^{N,m,\alpha}_{1,\mu, \lambda,t} $ is real valued differentiable function on $ (0, 1) $ satisfying the properties $ \Phi^{N,m,\alpha}_{1,\mu, \lambda,t}(0)=-1-2(1-\alpha)\left(\ln 2-1\right)<0$ and $ \lim\limits_{r\rightarrow 1}\Phi^{N,m,\alpha}_{1,\mu, \lambda,t}(r)=+\infty $, and hence the existence of the root $ R^{m, N, t}_{1,\mu, \lambda}(\alpha) $ is confirmed. Now in view of \eqref{e-2.11}, by the Intermediate Value Theorem, the root $ R^{m, N, t}_{1,\mu, \lambda}(\alpha) $ is unique. Therefore, we have 
 \begin{align} &\label{e-2.12}
 	\left(F_{\alpha}\left(R^{m, N, t}_{1,\mu, \lambda}(\alpha)\right)\right)^m-2(1-\alpha)\left(\ln\left(1-R^{m, N, t}_{1,\mu, \lambda}(\alpha)\right)+\sum_{n=1}^{N-1}\frac{\left(R^{m, N, t}_{1,\mu, \lambda}(\alpha)\right)^n}{n}\right)\\&\nonumber\quad+\mu G^N_{\alpha, t}\left(R^{m, N, t}_{1,\mu, \lambda}(\alpha)\right)-\lambda H_{\alpha,t}\left(R^{m, N, t}_{1,\mu, \lambda}(\alpha)\right)=1+2(1-\alpha)\left(\ln 2-1\right).
 \end{align}
In view of \eqref{e-2.3} and \eqref{e-2.9}, we see that $ S^f_{\mu, \lambda, m, N}(r) \leq {d}\left(f(0),\partial \mathbb{D}\right) $ holds.\vspace{1.2mm}

	\par In order to show that the constant $ R^{m, N, t}_{1,\mu, \lambda}(\alpha) $ is best possible, we consider the function $ f=f_{\alpha} $ which is defined by 
	\begin{equation*}
		f_{\alpha}(z)=z+\sum_{n=2}^{\infty}\frac{2(1-\alpha)z^n}{n}.
	\end{equation*}
	It is easy to see that $ f_{\alpha}\in\mathcal{P}^{0}_{\mathcal{H}}(\alpha) $ and  for $ f=f_{\alpha} $, in view of \eqref{e-2.1}, by a routine computation, it can be shown that
	\begin{equation}\label{e-2.13}
		d(f_{\alpha}(0),\partial f_{\alpha}(\mathbb{D}))=1+2(1-\alpha)(\ln 2 -1).
	\end{equation}
	For $ f=f_{\alpha} $ and $z=r>R^{m, N, t}_{1,\mu, \lambda}(\alpha) $, a simple computation using \eqref{e-2.12} and \eqref{e-2.13} shows that 
	\begin{align*} 
S^{f_{\alpha}}_{\mu, \lambda, m,  N}(r)&=\left(r+\sum_{n=2}^{\infty}\frac{2(1-\alpha)r^n}{n}\right)^m+\sum_{n=N}^{\infty}\frac{2(1-\alpha)r^n}{n}+\mu\; sgn(t)\sum_{n=2}^{t}\frac{4(1-\alpha)^2}{n^2}\frac{r^N}{1-r}\\&\quad+\lambda\left(1+\frac{r}{1-r}\right)\sum_{n=t+1}^{\infty}\frac{4(1-\alpha)^2r^{2n}}{n^2}\\&>\left(R^{m, N, t}_{1,\mu, \lambda}(\alpha)+\sum_{n=2}^{\infty}\frac{2(1-\alpha)\left(R^{m, N, t}_{1,\mu, \lambda}(\alpha)\right)^n}{n}\right)^m+\sum_{n=N}^{\infty}\frac{2(1-\alpha)\left(R^{m, N, t}_{1,\mu, \lambda}(\alpha)\right)^n}{n}\\&\quad+\lambda\left(1+\frac{R^{m, N, t}_{1,\mu, \lambda}(\alpha)}{1-R^{m, N, t}_{1,\mu, \lambda}(\alpha)}\right)\sum_{n=t+1}^{\infty}\frac{4(1-\alpha)^2\left(R^{m, N, t}_{1,\mu, \lambda}(\alpha)\right)^{2n}}{n^2}\\&\quad+\mu\; sgn(t)\sum_{n=2}^{t}\frac{4(1-\alpha)^2}{n^2}\frac{\left(R^{m, N, t}_{1,\mu, \lambda}(\alpha)\right)^N}{1-R^{m, N, t}_{1,\mu, \lambda}(\alpha)}\\&=\left(F_{\alpha}\left(R^{m, N, t}_{1,\mu, \lambda}(\alpha)\right)\right)^m-2(1-\alpha)\left(\ln\left(1-R^{m, N, t}_{1,\mu, \lambda}(\alpha)\right)+\sum_{n=1}^{N-1}\frac{\left(R^{m, N, t}_{1,\mu, \lambda}(\alpha)\right)^n}{n}\right)\\&\nonumber\quad+\mu G^N_{\alpha, t}\left(R^{m, N, t}_{1,\mu, \lambda}(\alpha)\right)-\lambda H_{\alpha,t}\left(R^{m, N, t}_{1,\mu, \lambda}(\alpha)\right)\\&\nonumber= 1+2(1-\alpha)(\ln 2-1)\\&=d(f_{\alpha}(0),\partial f_{\alpha}(\mathbb{D})).
	\end{align*}
	Hence, the radius $ R^{m, N, t}_{1,\mu, \lambda}(\alpha) $ is best possible. This completes the proof.
\end{proof}	
\subsection{Refined Bohr-type inequality for the class $ \mathcal{P}^{0}_{\mathcal{H}}(M) $}
\par The main aim of this paper is to establish several refined Bohr-Rogosinski inequalities, finding the corresponding sharp radius for the class $ \mathcal{P}^{0}_{\mathcal{H}}(M) $ which has been studied by Ghosh and Vasudevarao in \cite{Ghosh-Vasudevarao-BAMS-2020} 
$$\mathcal{P}^{0}_{\mathcal{H}}(M)=\{f=h+\overline{g} \in \mathcal{H}_{0}: \real (zh^{\prime\prime}(z))> -M+|zg^{\prime\prime}(z)|, \; z \in \mathbb{D}\; \mbox{and }\; M>0\}.$$

\par To study Bohr inequality and Bohr radius for functions in $ \mathcal{P}^{0}_{\mathcal{H}}(M) $, we require the coefficient bounds and growth estimate of functions in $ \mathcal{P}^{0}_{\mathcal{H}}(M) $. We have the following result on the coefficient bounds and growth estimate for functions in $ \mathcal{P}^{0}_{\mathcal{H}}(M) $.
\begin{lem} \label{lem-2.19} \cite{Ghosh-Vasudevarao-BAMS-2020}
	Let $f=h+\overline{g}\in \mathcal{P}^{0}_{\mathcal{H}}(M)$ be given by \eqref{e-1.3} for some $M>0$. Then for $n\geq 2,$ 
	\begin{enumerate}
		\item[(i)] $\displaystyle |a_n| + |b_n|\leq \frac {2M}{n(n-1)}; $\\[2mm]
		
		\item[(ii)] $\displaystyle ||a_n| - |b_n||\leq \frac {2M}{n(n-1)};$\\[2mm]
		
		\item[(iii)] $\displaystyle |a_n|\leq \frac {2M}{n(n-1)}.$
	\end{enumerate}
	The inequalities  are sharp with extremal function   $f$ given by 
	$f^{\prime}(z)=1-2M\, \ln\, (1-z) .$	
\end{lem}
\begin{lem}\cite{Ghosh-Vasudevarao-BAMS-2020}\label{lem-2.20}
	Let $f \in \mathcal{P}^{0}_{\mathcal{H}}(M)$ be given by \eqref{e-1.3}. Then 
	\begin{equation} \label{e-2.15}
		|z| +2M \sum\limits_{n=2}^{\infty} \dfrac{(-1)^{n-1}|z|^{n}}{n(n-1)} \leq |f(z)| \leq |z| + 2M \sum\limits_{n=2}^{\infty} \dfrac{|z|^{n}}{n(n-1)}.
	\end{equation}
	Both  inequalities are sharp for the function $f_{M}$ given by $f_{M}(z)=z+ 2M \sum\limits_{n=2}^{\infty} \dfrac{z^n}{n(n-1)}.
	$
\end{lem}
\begin{thm}\label{th-2.2}
Let $ f\in \mathcal{P}^{0}_{\mathcal{H}}(M) $ be given by \eqref{e-1.3} and $ 0\leq M <1/(2(\ln 4-1))$. Then for $ \mu,\; \lambda\in\mathbb{R}_{\geq 0} $ and $ N\geq 5 $, we have $ S^f_{\mu, \lambda, m, N}(r) \leq {d}\left(f(0),\partial \mathbb{D}\right) $ for $ |z|=r\leq R^{m, N, t}_{2,\mu, \lambda}(M) $, where $ R^{m, N, t}_{2,\mu, \lambda}(M) $ is the unique root of the equation $ 	\Phi^{N,m,M}_{2,\mu, \lambda,t}(r)=0 $ in $ (0,1) $ and
\begin{align*}
\Phi^{N,m,M}_{2,\mu, \lambda,t}(r):\nonumber&= \left(G_{M}(r)\right)^m+2M\left(r+(1-r)\ln(1-r)-\sum_{n=2}^{N-1}\frac{r^n}{n(n-1)}\right)+\mu \Phi^N_{M,t}(r)\\&\quad+4M^2\lambda\left(1+\frac{r}{1-r}\right) G_{2,t}(r)-1-2M\left(1-2\ln 2\right),
\end{align*}
	and
	\[
	\begin{cases}
G_{M}(r):=\displaystyle r+2M\sum_{n=2}^{\infty}\frac{r^n}{n(n-1)}=r+2M\left(r+(1-r)\ln (1-r)\right)\vspace{1.4mm}\\
\Phi^N_{M,t}(r):=\displaystyle \frac{4M^2r^N}{1-r}sgn(t)\sum_{n=1}^{t}\frac{1}{n^2(n-1)^2}\vspace{1.4mm}\\
G_{2,t}(r):=\displaystyle\left(r^2+1\right){\rm Li_2\left(r^2\right)}+2\left(r^2-1\right)\ln\left(1-r^2\right)-3r^2-\sum_{n=2}^{t}\frac{r^{2n}}{n^2(n-1)^2}
	\end{cases}
	\]
	The constant $ R^{m, N, t}_{2,\mu, \lambda}(M) $ is best possible.
\end{thm}
\begin{rem} For the triplets $ (0, 0, 1)$, $(0, 0, 2)$, $(1, 1, 1)$, $(1, 1, 2) $ corresponding to $ (\lambda, \mu, m) $, the inequality $ S^f_{\mu, \lambda, m, N}(r) \leq {d}\left(f(0),\partial \mathbb{D}\right) $ is harmonic analog of \eqref{e-11.44}, \eqref{e-11.55}, \eqref{e-11.88} and \eqref{e-11.99}, respectively, for the class $ \mathcal{P}^{0}_{\mathcal{H}}(M) $.
\end{rem}
We define here some notations
\[
\begin{cases}
\mathcal{L}^M_1(r):=r^2+4M^2[\left(r^2+1\right){\rm Li_2\left(r^2\right)}+2\left(r^2-1\right)\ln\left(1-r^2\right)-3r^2]\\
\mathcal{L}^M_2(m, r):=(G_M(r))^m-1-2M(1-2\ln 2)\\
\mathcal{L}_3(r):=r+(1-r)\ln(1-r).
\end{cases}
\]
As a corollary of Theorem \ref{th-2.2}, we obtain the following result in which the cases for $ N=1, 2, 3, 4 $ are discussed.
\begin{cor} Let $ f\in \mathcal{P}^{0}_{\mathcal{H}}(M) $ be given by \eqref{e-1.3} and $ 0\leq M <1/(2(\ln 4-1))$, $ \mu,\; \lambda\in\mathbb{R}_{\geq 0} $.
\begin{enumerate}
\item[(i)] If $ N=1 $, then $ S^f_{\mu, \lambda, m, 1}(r)\leq {d}\left(f(0),\partial \mathbb{D}\right) $
for $ |z|=r\leq R^{m, 1, 0}_{2,\mu, \lambda}(M) $, where $ R^{m, 1, 0}_{2,\mu, \lambda}(M) $ is the unique root in $ (0,1) $ of the equation 
\begin{align*}
r+2M\mathcal{L}_3(r)+\mathcal{L}^M_2(m, r)+\lambda\left(1+\frac{r}{1-r}\right)\mathcal{L}^M_1(r)=0.
\end{align*}
\item[(ii)] If $ N=2 $, then $ S^f_{\mu, \lambda, m, 2}(r)\leq {d}\left(f(0),\partial \mathbb{D}\right) $
for $ |z|=r\leq R^{m, 2, 0}_{2,\mu, \lambda}(M) $, where $ R^{m, 2, 0}_{2,\mu, \lambda}(M) $ is the unique root in $ (0,1) $ of the equation 
\begin{align*}
2M\mathcal{L}_3(r)+\mathcal{L}^M_2(m, r)+\lambda\left(1+\frac{r}{1-r}\right)\mathcal{L}^M_1(r)=0.
\end{align*}
\item[(iii)] If $ N=3 $, then $ S^f_{\mu, \lambda, m, 3}(r)\leq {d}\left(f(0),\partial \mathbb{D}\right) $
for $ |z|=r\leq R^{m, 3, 1}_{2,\mu, \lambda}(M) $, where $ R^{m, 3, 1}_{2,\mu, \lambda}(M) $ is the unique root in $ (0,1) $ of the equation 
\begin{align*}
2M\left(\mathcal{L}_3(r)-\frac{r^2}{2}\right)+\mathcal{L}^M_2(m, r)+\mu\frac{r^3}{1-r}+\lambda\left(1+\frac{r}{1-r}\right)\left(\mathcal{L}^M_1(r)-r^2\right)=0.
\end{align*}
\item[(iv)] If $ N=4 $, then $ S^f_{\mu, \lambda, m, 4}(r)\leq {d}\left(f(0),\partial \mathbb{D}\right) $
for $ |z|=r\leq R^{m, 4, 1}_{2,\mu, \lambda}(M) $, where $ R^{m, 4, 1}_{2,\mu, \lambda}(M) $ is the unique root in $ (0,1) $ of the equation 
\begin{align*}
2M\left(\mathcal{L}_3(r)-\frac{r^2}{2}-\frac{r^3}{6}\right)+\mathcal{L}^M_2(m, r)+\mu\frac{r^4}{1-r}+\lambda\left(1+\frac{r}{1-r}\right)\left(\mathcal{L}^M_1(r)-r^2\right)=0.
\end{align*}
\end{enumerate}
The constants $ R^{m, 1, 0}_{2,\mu, \lambda}(M) $, $ R^{m, 2, 0}_{2,\mu, \lambda}(M) $, $ R^{m, 3, 1}_{2,\mu, \lambda}(M) $ and $ R^{m, 4, 1}_{2,\mu, \lambda}(M) $ are best possible.
\end{cor}

\begin{proof}[\bf Proof of Theorem \ref{th-2.2}]
By a straightforward computation, it can be shown that
\[
\begin{cases}
	\displaystyle\sum_{n=2}^{\infty}\frac{r^n}{n(n-1)}=r+(1-r)\ln (1-r)\vspace{2mm}\\
	\displaystyle\sum_{n=N}^{\infty}\frac{r^n}{n(n-1)}=r+(1-r)\ln(1-r)-\sum_{n=2}^{N-1}\frac{r^n}{n(n-1)}\vspace{2mm}\\
	\displaystyle \sum_{n=2}^{\infty}\frac{r^{2n}}{n^2(n-1)^2}=\left(r^2+1\right){\rm Li_2\left(r^2\right)}+2\left(r^2-1\right)\ln\left(1-r^2\right)-3r^2\vspace{2mm}\\
	\displaystyle \sum_{n=t+1}^{\infty}\frac{r^{2n}}{n^2(n-1)^2}=\left(r^2+1\right){\rm Li_2\left(r^2\right)}+2\left(r^2-1\right)\ln\left(1-r^2\right)-3r^2-\sum_{n=2}^{t}\frac{r^{2n}}{n^2(n-1)^2}.\vspace{2mm}\\
	
\end{cases}
\]
	In view of the above computations, using Lemma \ref{lem-2.20}, we obtain
	\begin{align}\label{e-2.17}
		|f(z)|\leq r+2M\sum_{n=2}^{\infty}\frac{r^n}{n(n-1)}=r+2M\left(r+(1-r)\ln (1-r)\right)=:G_{M}(r).
	\end{align}
Since $ N\geq 3 $, by the above computations, using Lemma \ref{lem-2.19}, a simple computation shows that
\begin{align}\label{e-2.18}
	\sum_{n=N}^{\infty}\left(|a_n|+|b_n|\right)r^n\leq 2M\left(r+(1-r)\ln(1-r)-\sum_{n=2}^{N-1}\frac{r^n}{n(n-1)}\right),
\end{align}
\begin{align}\label{e-2.19}
	\sum_{n=t+1}^{\infty}\left(|a_n|+|b_n|\right)^2r^{2n}\leq 4M^2\sum_{n=t+1}^{\infty}\frac{r^{2n}}{n^2(n-1)^2}=4M^2 G_{2,t}(r),
\end{align}
where 
\begin{align*}
	G_{2,t}(r):=\left(r^2+1\right){\rm Li_2\left(r^2\right)}+2\left(r^2-1\right)\ln\left(1-r^2\right)-3r^2-\sum_{n=2}^{t}\frac{r^{2n}}{n^2(n-1)^2}.
\end{align*}
and 
\begin{align}\label{e-2.20}
	sgn(t)\sum_{n=1}^{t}\left(|a_n|+|b_n|\right)^2\frac{r^N}{1-r}\leq sgn(t)\frac{4M^2r^N}{1-r}\sum_{n=1}^{t}\frac{1}{n^2(n-1)^2}=:\Phi^N_{M,t}(r).
\end{align}
Thus, using \eqref{e-2.17} to \eqref{e-2.20}, it is easy to see that
\begin{align} \label{e-2.21}
	S^f_{\mu, \lambda, m, N}(r)&\leq \left(r+2M\sum_{n=2}^{\infty}\frac{r^n}{n(n-1)}\right)^m+2M\sum_{n=N}^{\infty}\frac{r^n}{n(n-1)}+\mu \;sgn(t)\frac{4M^2r^N}{1-r}\sum_{n=1}^{t}\frac{1}{n^2(n-1)^2}\\&\nonumber\quad+\lambda\left(1+\frac{r}{1-r}\right)\sum_{n=t+1}^{\infty}\frac{4M^2r^{2n}}{n^2(n-1)^2}\\&\nonumber\leq  \left(G_{M}(r)\right)^m+2M\left(r+(1-r)\ln(1-r)-\sum_{n=2}^{N-1}\frac{r^n}{n(n-1)}\right)+\mu \Phi^N_{M,t}(r)\\&\nonumber\quad+4M^2\lambda\left(1+\frac{r}{1-r}\right) G_{2,t}(r)\\&\nonumber\leq 1+2M\left(1-2\ln 2\right)
\end{align} 
for $ |z|=r\leq R^{m, N, t}_{2,\mu, \lambda}(M) $, where $ R^{m, N, t}_{2,\mu, \lambda}(M) $ is the smallest root of $ \Phi^{N,m,M}_{2,\mu, \lambda,t}(r)=0 $ in $ (0, 1) $. By the similar argument as in proof of Theorem \ref{th-2.6}, it can be shown that $ R^{m, N, t}_{2,\mu, \lambda}(M)  $ is the unique root of the equation $ \Phi^{N,m,M}_{2,\mu, \lambda,t}(r)=0 $ in $ (0, 1) $. Therefore, we have
\begin{align}\label{e-2.22}&  2M\left(R^{m, N, t}_{2,\mu, \lambda}(M)+\left(1-R^{m, N, t}_{2,\mu, \lambda}(M)\right)\ln\left(1-R^{m, N, t}_{2,\mu, \lambda}(M)\right)-\sum_{n=2}^{N-1}\frac{\left(R^{m, N, t}_{2,\mu, \lambda}(M)\right)^n}{n(n-1)}\right)\\&\nonumber+\mu \Phi^N_{M,t}\left(R^{m, N, t}_{2,\mu, \lambda}(M)\right)+4M^2\lambda\left(1+\frac{R^{m, N, t}_{2,\mu, \lambda}(M)}{1-R^{m, N, t}_{2,\mu, \lambda}(M)}\right) G_{2,t}\left(R^{m, N, t}_{2,\mu, \lambda}(M)\right)\\&\nonumber+\left(G_{M}\left(R^{m, N, t}_{2,\mu, \lambda}(M)\right)\right)^m=1+2M\left(1-2\ln 2\right).
\end{align}
Since $f(0)=0$, then in view of Lemma \ref{lem-2.20} and \eqref{e-2.1}, a simple computation shows that
\begin{equation} \label{e-2.23}
	d(f(0), \partial f(\mathbb{D})) \geq 1+\sum\limits_{n=2}^{\infty}  2M \dfrac{(-1)^{n-1}}{n(n-1)}=1+2M\left(1-2\ln 2\right).
\end{equation}
Therefore, in view of \eqref{e-2.21} and \eqref{e-2.23}, we see that $ S^f_{\mu, \lambda, m, N}(r) \leq {d}\left(f(0),\partial \mathbb{D}\right) $ holds. Now it remains to show that the constant $ R^{m, N, t}_{2,\mu, \lambda}(M) $ is best possible. Henceforth, we consider the function $ f=f_{M} $ defined by 
\begin{equation*}
	f_{M}(z)=z+\sum_{n=2}^{\infty}\frac{2Mz^n}{n(n-1)}.
\end{equation*}
It is easy to see that $ f_{M}\in\mathcal{P}^{0}_{\mathcal{H}}(M) $ and  for $ f=f_{M} $, in view of \eqref{e-2.1}, by a routine computation, we can show that
\begin{equation}\label{e-2.2424}
	d(f_{M}(0),\partial f_{M}(\mathbb{D}))=1+2M\left(1-2\ln 2\right).
\end{equation}
For $ f=f_{M} $ and $z=r>R^{m, N, t}_{2,\mu, \lambda}(M) $, by the similar argument as in the proof of Theorem \ref{th-2.6}, using \eqref{e-2.21} 
to \eqref{e-2.2424}, it can be shown that 
\begin{align*} 
	S^{f_M}_{\mu, \lambda, m, N}(r)&>2M\left(R^{m, N, t}_{2,\mu, \lambda}(M)+\left(1-R^{m, N, t}_{2,\mu, \lambda}(M)\right)\ln\left(1-R^{m, N, t}_{2,\mu, \lambda}(M)\right)-\sum_{n=2}^{N-1}\frac{\left(R^{m, N, t}_{2,\mu, \lambda}(M)\right)^n}{n(n-1)}\right)\\&\nonumber\quad+\mu \Phi^N_{M,t}\left(R^{m, N, t}_{2,\mu, \lambda}(M)\right)+4M^2\lambda\left(1+\frac{R^{m, N, t}_{2,\mu, \lambda}(M)}{1-R^{m, N, t}_{2,\mu, \lambda}(M)}\right) G_{2,t}\left(R^{m, N, t}_{2,\mu, \lambda}(M)\right)\\&\nonumber\quad+\left(G_{M}\left(R^{m, N, t}_{2,\mu, \lambda}(M)\right)\right)^m\\&=d(f_{M}(0),\partial f_{M}(\mathbb{D})).
\end{align*}
Therefore, $ R^{m, N, t}_{2,\mu, \lambda}(M) $ is best possible. This completes the proof.
\end{proof}	
\subsection{Refined Bohr-type inequality for the class $ \mathcal{W}^{0}_{\mathcal{H}}(\alpha) $}
\par In $ 1977 $, Chichra \cite{Chichra-PAMS-1977} introduced the class $ \mathcal{W}(\alpha) $ consisting of normalized analytic functions $ h $, satisfying the condition $ {\rm Re} \left(h^{\prime}(z)+\alpha zh^{\prime\prime}(z)\right)>0 $ for $ z\in\mathbb{D} $ and $ \alpha\geq 0 $. Moreover, Chichra \cite{Chichra-PAMS-1977} has shown that functions in the class $ \mathcal{W}(\alpha) $ constitute a subclass of close-to-convex functions in $ \mathbb{D} $. In $ 2014 $, Nagpal and Ravichandran \cite{Nagpal-Ravinchandran-2014-JKMS} studied the following class 
\begin{equation*}
	\mathcal{W}^{0}_{\mathcal{H}}=\{f=h+\bar{g}\in \mathcal{H} :  {\rm Re}\left(h^{\prime}(z)+zh^{\prime\prime}(z)\right) > |g^{\prime}(z)+zg^{\prime\prime}(z)|\;\; \mbox{for}\; z\in\mathbb{D}\}
\end{equation*}
and obtained the coefficient bounds for the functions in the class $ \mathcal{W}^{0}_{\mathcal{H}} $. In $ 2019 $, Ghosh and Vasudevarao studied the class $ \mathcal{W}^{0}_{\mathcal{H}}(\alpha) $, where 
\begin{equation*}
	\mathcal{W}^{0}_{\mathcal{H}}(\alpha)=\{f=h+\bar{g}\in \mathcal{H} :  {\rm Re}\left(h^{\prime}(z)+\alpha zh^{\prime\prime}(z)\right) > |g^{\prime}(z)+\alpha zg^{\prime\prime}(z)|\;\; \mbox{for}\; z\in\mathbb{D}\}.
\end{equation*}
From the following results, it is easy to see that functions in the class $ \mathcal{W}^{0}_{\mathcal{H}}(\alpha) $ are univalent for $ \alpha\geq 0 $, and they are closely related to functions in $ \mathcal{W}(\alpha) $.
\begin{lem}\cite{Nirupam-MonatsMath-2019}\label{lem-2.15}
	The harmonic mapping $ f=h+\bar{g} $ belongs to $ \mathcal{W}^{0}_{\mathcal{H}}(\alpha) $ if, and only if, the analytic function $ F=h+\epsilon g $ belongs to $ \mathcal{W}(\alpha) $ for each $ |\epsilon|=1. $ 
\end{lem}
The coefficient bounds and the sharp growth estimates for functions in the class $ \mathcal{W}^{0}_{\mathcal{H}}(\alpha) $ have been studied in \cite{Nirupam-MonatsMath-2019}.
\begin{lem}\cite{Nirupam-MonatsMath-2019}\label{lem-2.6}
	Let $ f\in \mathcal{W}^{0}_{\mathcal{H}}(\alpha) $ for $ \alpha\geq 0 $ and be of the form \eqref{e-1.3}. Then for any $ n\geq 2 $,
	\begin{enumerate}
		\item[(i)] $ |a_n|+|b_n|\leq \displaystyle\frac{2}{\alpha n^2+(1-\alpha)n} $;\vspace{1.5mm}
		\item[(ii)] $  ||a_n|-|b_n||\leq \displaystyle\frac{2}{\alpha n^2+(1-\alpha)n} $; \vspace{1.5mm}
		\item[(iii)] $ |a_n|\leq \displaystyle\frac{2}{\alpha n^2+(1-\alpha)n} $.
	\end{enumerate}
	All these inequalities are sharp for the function $ f=f^*_{\alpha} $ given by 
	\begin{equation}\label{e-2.24}
		f^*_{\alpha}(z)=z+\sum_{n=2}^{\infty}\frac{2z^n}{\alpha n^2+(1-\alpha)n}.
	\end{equation}
\end{lem}
\begin{lem}\cite{Nirupam-MonatsMath-2019}\label{lem-2.7}
	Let $ f\in \mathcal{W}^{0}_{\mathcal{H}}(\alpha) $ and be of the form \eqref{e-1.3} with $ 0\leq\alpha<1 $. Then
	\begin{equation}\label{e-2.25}
		|z|+\sum_{n=2}^{\infty}\frac{2(-1)^{n-1}|z|^n}{\alpha n^2+(1-\alpha)n}\leq |f(z)|\leq |z|+\sum_{n=2}^{\infty}\frac{2|z|^n}{\alpha n^2+(1-\alpha)n}.
	\end{equation}
	Both the inequalities are sharp for the function  $ f=f^*_{\alpha} $ given by \eqref{e-2.24}.
\end{lem}
We prove the following sharp refinement of the Bohr-Rogosinski inequality for functions in the class $ \mathcal{W}^{0}_{\mathcal{H}}(\alpha) $.  
\begin{thm}\label{th-2.3}
	Let $ f\in \mathcal{W}^{0}_{\mathcal{H}}(\alpha) $ be given by \eqref{e-1.3} and $ 0\leq \alpha <1 $. Then for $ \mu,\; \lambda\in\mathbb{R}_{\geq 0} $ and $ N\geq 5 $, we have $ S^f_{\mu, \lambda, m, N}(r)\leq {d}\left(f(0),\partial \mathbb{D}\right) $
	for $ |z|=r\leq R^{m, N, t}_{3,\mu, \lambda}(\alpha) $, where $ R^{m, N, t}_{3,\mu, \lambda}(\alpha) $ is the unique root in $ (0,1) $ of the equation $ \Phi^{N,m,\alpha}_{3,\mu, \lambda,t}(r)=0 $, where
	\begin{align*}
	\Phi^{N,m,\alpha}_{3,\mu, \lambda,t}(r):=&\left(r+\sum_{n=2}^{\infty}\frac{2r^n}{\alpha n^2+(1-\alpha)n}\right)^m+\sum_{n=N}^{\infty}\frac{2r^n}{\alpha n^2+(1-\alpha)n}+\frac{r^N}{1-r}\sum_{n=1}^{t}\frac{4\mu\; sgn(t)}{\left(\alpha n^2+(1-\alpha)n\right)^2}\\&\quad+\lambda\left(1+\frac{r}{1-r}\right)\sum_{n=t+1}^{\infty}\frac{4r^{2n}}{\left(\alpha n^2+(1-\alpha)n\right)^2}-1-\sum_{n=2}^{\infty}\frac{2(-1)^{n-1}}{\alpha n^2+(1-\alpha)n}.
	\end{align*}
	The constant $ R^{m, N, t}_{3,\mu, \lambda}(\alpha) $ is best possible.
\end{thm}
\begin{rem} For the triplets $ (0, 0, 1)$, $(0, 0, 2)$, $(1, 1, 1)$, $(1, 1, 2) $ corresponding to $ (\lambda, \mu, m) $, the inequality $ S^f_{\mu, \lambda, m, N}(r)\leq {d}\left(f(0),\partial \mathbb{D}\right) $ is harmonic analog of \eqref{e-11.44}, \eqref{e-11.55}, \eqref{e-11.88} and \eqref{e-11.99}, respectively, for the class $ \mathcal{W}^{0}_{\mathcal{H}}(\alpha) $.
\end{rem}
We define 
\begin{align*}
	\mathcal{C}_{m, N}(r):&=\left(r+\sum_{n=2}^{\infty}\frac{2r^n}{\alpha n^2+(1-\alpha)n}\right)^m+\sum_{n=N}^{\infty}\frac{2r^n}{\alpha n^2+(1-\alpha)n}\\&\quad-1-\sum_{n=2}^{\infty}\frac{2(-1)^{n-1}}{\alpha n^2+(1-\alpha)n}.
\end{align*}
As a corollary of Theorem \ref{th-2.3}, we obtain the following result exploring the situation when $ N=1, 2, 3, 4 $.
\begin{cor} Let $ f\in \mathcal{W}^{0}_{\mathcal{H}}(\alpha) $ be given by \eqref{e-1.3} and $ 0\leq \alpha <1 $ and $ \mu,\; \lambda\in\mathbb{R}_{\geq 0} $.
	\begin{enumerate}
		\item[(i)] If $ N=1 $, then $ S^f_{\mu, \lambda, m, 1}(r)\leq {d}\left(f(0),\partial \mathbb{D}\right) $
		for $ |z|=r\leq R^{m, 1, 0}_{3,\mu, \lambda}(\alpha) $, where $ R^{m, 1, 0}_{3,\mu, \lambda}(\alpha) $ is the unique root in $ (0,1) $ of the equation 
		\begin{align*}
			\mathcal{C}_{m, 1}(r)+\lambda\left(1+\frac{r}{1-r}\right)\sum_{n=1}^{\infty}\frac{4r^{2n}}{\left(\alpha n^2+(1-\alpha)n\right)^2}=0.
		\end{align*}
		\item[(ii)] If $ N=2 $, then $ S^f_{\mu, \lambda, m, 2}(r)\leq {d}\left(f(0),\partial \mathbb{D}\right) $
		for $ |z|=r\leq R^{m, 2, 0}_{3,\mu, \lambda}(\alpha) $, where $ R^{m, 2, 0}_{3,\mu, \lambda}(\alpha) $ is the unique root in $ (0,1) $ of the equation 
		\begin{align*}
			\mathcal{C}_{m, 2}(r)+\lambda\left(1+\frac{r}{1-r}\right)\sum_{n=1}^{\infty}\frac{4r^{2n}}{\left(\alpha n^2+(1-\alpha)n\right)^2}=0.
		\end{align*}
		\item[(iii)] If $ N=3 $, then $ S^f_{\mu, \lambda, m, 3}(r)\leq {d}\left(f(0),\partial \mathbb{D}\right) $
		for $ |z|=r\leq R^{m, 3, 1}_{3,\mu, \lambda}(\alpha) $, where $ R^{m, 3, 1}_{3,\mu, \lambda}(\alpha) $ is the unique root in $ (0,1) $ of the equation 
		\begin{align*}
			\mathcal{C}_{m, 3}(r)+\frac{r^3}{1-r}+\lambda\left(1+\frac{r}{1-r}\right)\sum_{n=2}^{\infty}\frac{4r^{2n}}{\left(\alpha n^2+(1-\alpha)n\right)^2}=0.
		\end{align*}
		\item[(iv)] If $ N=4 $, then $ S^f_{\mu, \lambda, m, 4}(r)\leq {d}\left(f(0),\partial \mathbb{D}\right) $
		for $ |z|=r\leq R^{m, 4, 1}_{3,\mu, \lambda}(\alpha) $, where $ R^{m, 4, 1}_{3,\mu, \lambda}(\alpha) $ is the unique root in $ (0,1) $ of the equation 
		\begin{align*}
			\mathcal{C}_{m, 4}(r)+\frac{r^4}{1-r}+\lambda\left(1+\frac{r}{1-r}\right)\sum_{n=2}^{\infty}\frac{4r^{2n}}{\left(\alpha n^2+(1-\alpha)n\right)^2}=0.
		\end{align*}
	\end{enumerate}
The constants $ R^{m, 1, 0}_{3,\mu, \lambda}(\alpha) $, $ R^{m, 2, 0}_{3,\mu, \lambda}(\alpha) $, $ R^{m, 3, 1}_{3,\mu, \lambda}(\alpha) $ and $ R^{m, 4, 1}_{3,\mu, \lambda}(\alpha) $ are best possible.
\end{cor}
\begin{proof}[\bf Proof of Theorem \ref{th-2.3}]
	Since $ f\in \mathcal{W}^{0}_{\mathcal{H}}(\alpha)$, in view of Lemma \ref{lem-2.7} and \eqref{e-2.1}, it is easy to see that
	\begin{align}\label{e-2.26}
	d(f(0), \partial f(\mathbb{D}))\geq 1+\sum_{n=2}^{\infty}\frac{2(-1)^{n-1}}{\alpha n^2+(1-\alpha)n}.
\end{align}
An easy computation using Lemma \ref{lem-2.6} shows that
\begin{align}\label{e-2.27} 
	S^f_{\mu, \lambda, m, N}(r)&\leq \left(r+\sum_{n=2}^{\infty}\frac{2r^n}{\alpha n^2+(1-\alpha)n}\right)^m+\sum_{n=N}^{\infty}\frac{2r^n}{\alpha n^2+(1-\alpha)n}\\&\nonumber\quad+\frac{r^{N}}{1-r}\sum_{n=1}^{t}\frac{4\mu\; sgn(t)}{\left(\alpha n^2+(1-\alpha)n\right)^2}+\lambda\left(1+\frac{r}{1-r}\right)\sum_{n=t+1}^{\infty}\frac{4r^{2n}}{\left(\alpha n^2+(1-\alpha)n\right)^2}\\&\nonumber \leq 1+\sum_{n=2}^{\infty}\frac{2(-1)^{n-1}}{\alpha n^2+(1-\alpha)n}
\end{align} 
if $ |z|=r\leq R^{m, N, t}_{3,\mu, \lambda}(\alpha) $ is the smallest root of the equation 
\begin{align*}
	\Phi^{N,m,\alpha}_{3,\mu, \lambda,t}(r)=&\left(r+\sum_{n=2}^{\infty}\frac{2r^n}{\alpha n^2+(1-\alpha)n}\right)^m+\sum_{n=N}^{\infty}\frac{2r^n}{\alpha n^2+(1-\alpha)n}+\frac{r^N}{1-r}\sum_{n=1}^{t}\frac{4\mu\; sgn(t)}{\left(\alpha n^2+(1-\alpha)n\right)^2}\\&\quad+\lambda\left(1+\frac{r}{1-r}\right)\sum_{n=t+1}^{\infty}\frac{4r^{2n}}{\left(\alpha n^2+(1-\alpha)n\right)^2}-1-\sum_{n=2}^{\infty}\frac{2(-1)^{n-1}}{\alpha n^2+(1-\alpha)n}=0
\end{align*}
in $ (0,1) $. It is easy to see that $ \Phi^{N,m,\alpha}_{3,\mu, \lambda,t}(r) $ is a real valued differentiable function satisfying 
\begin{align*}
		\Phi^{N,m,\alpha}_{3,\mu, \lambda,t}(0)=-1-\sum_{n=2}^{\infty}\frac{2(-1)^{n-1}}{\alpha n^2+(1-\alpha)n}<0\;\mbox{and}\; \lim\limits_{r\ra 1}\Phi^{N,m,\alpha}_{3,\mu, \lambda,t}(r)=+\infty
\end{align*}
and by a simple computation it can be shown that 
\begin{align*}
	\frac{d}{dr}\left(\Phi^{N,m,\alpha}_{3,\mu, \lambda,t}(r)\right)>0\;\; \mbox{for}\;\; r\in (0, 1).
\end{align*}
By adopting the similar arguments as applied in the proof of Theorems \ref{th-2.6} and \ref{th-2.2}, we can prove the existence and uniqueness of the root $ R^{m, N, t}_{3,\mu, \lambda}(\alpha) $. \vspace{1.2mm}
 Thus, we must have
	\begin{align}\label{e-2.28}
	&\left(R^{m, N, t}_{3,\mu, \lambda}(\alpha)+\sum_{n=2}^{\infty}\frac{2\left(R^{m, N, t}_{3,\mu, \lambda}(\alpha)\right)^n}{\alpha n^2+(1-\alpha)n}\right)^m+\sum_{n=N}^{\infty}\frac{2\left(R^{m, N, t}_{3,\mu, \lambda}(\alpha)\right)^n}{\alpha n^2+(1-\alpha)n}\\&\nonumber\quad+\frac{\left(R^{m, N, t}_{3,\mu, \lambda}(\alpha)\right)^N}{1-R^{m, N, t}_{3,\mu, \lambda}(\alpha)}\sum_{n=1}^{t}\frac{4\mu\; sgn(t)}{\left(\alpha n^2+(1-\alpha)n\right)^2}+\lambda\left(1+\frac{R^{m, N, t}_{3,\mu, \lambda}(\alpha)}{1-R^{m, N, t}_{3,\mu, \lambda}(\alpha)}\right)\sum_{n=t+1}^{\infty}\frac{4\left(R^{m, N, t}_{3,\mu, \lambda}(\alpha)\right)^{2n}}{\left(\alpha n^2+(1-\alpha)n\right)^2}\\&\nonumber=1+\sum_{n=2}^{\infty}\frac{2(-1)^{n-1}}{\alpha n^2+(1-\alpha)n}.
\end{align}
In view of \eqref{e-2.26} and \eqref{e-2.27}, we see that $ S^f_{\mu, \lambda, m, N}(r)\leq {d}\left(f(0),\partial \mathbb{D}\right) $ holds. To show that the constant $ R^{m, N, t}_{3,\mu, \lambda}(\alpha) $ is best possible, and henceforth, we consider the function $ f=f^{*}_{\alpha} $. A simple computation shows that
\begin{align*}
	d(f^{*}_{\alpha}(0), \partial f^{*}_{\alpha}(\mathbb{D}))=1+\sum_{n=2}^{\infty}\frac{2(-1)^{n-1}}{\alpha n^2+(1-\alpha)n}.
\end{align*}
By the similar arguments as used in the proof of Theorems \ref{th-2.6} and \ref{th-2.2}, using \eqref{e-2.28}, it can be easily shown for $ z=r> R^{m, N, t}_{3,\mu, \lambda}(\alpha)$ that $ S^{f^{*}_{\alpha}}_{\mu, \lambda, m, N}(r)> d(f^{*}_{\alpha}(0), \partial f^{*}_{\alpha}(\mathbb{D})). $ Therefore, $ R^{m, N, t}_{3,\mu, \lambda}(\alpha) $ is best possible. This completes the proof.
\end{proof}	

\noindent{\bf Acknowledgment:} The author wish to thank the anonymous referees for their helpful suggestions and comments to enhance the clarity and presentation of the paper. \\
\vspace{1.2mm}

\noindent\textbf{Compliance of Ethical Standards:}\\

\noindent\textbf{Conflict of interest.} The authors declare that there is no conflict  of interest regarding the publication of this paper.\vspace{1.5mm}

\noindent\textbf{Data availability statement.}  Data sharing not applicable to this article as no datasets were generated or analysed during the current study.

\end{document}